\documentclass[a4paper,12pt,reqno]{amsart}
\usepackage{amsmath,amssymb,amsthm}
\usepackage{mathrsfs}
\theoremstyle{plain}
\newtheorem{thm}{Theorem}[section]
\newtheorem{corollary}[thm]{Corollary}

\newtheorem{proposition}[thm]{Proposition}

\theoremstyle{definition}
\newtheorem{definition}[thm]{Definition}

\theoremstyle{remark}

\numberwithin{equation}{section}
\begin{document}
\newcommand{\ici}[1]{\stackrel{\circ}{#1}}
\begin{center}{\bf{ Lacunary Arithmetic convergence}}

\vspace{.5cm}
Taja Yaying$^{1}$, Bipan Hazarika$^{\ast 2}$ \\
$^{1}$Department of Mathematics, Dera Natung Govt. College, Itanagar-791 111, Arunachal Pradesh, India\\
$^{2}$Department of Mathematics, Rajiv Gandhi University, Rono Hills, Doimukh-791 112, Arunachal Pradesh, India\\
Email: tajayaying20@gmail.com;  bh\_rgu@yahoo.co.in
\end{center}
\title{}
\author{}
\thanks{$^\ast$The corresponding author.}
\date{\today} 

\begin{abstract} In this article we introduce and study the lacunary arithmetic convergent sequence space $AC_{\theta}$. Using the idea of strong Ces\`{a}ro summable sequence and arithmetic convergence we define $AC_{\sigma_1}$ and study the relations between $AC_{\theta}$ and $AC_{\sigma_1}$. Finally using modulus function we define $AC_{\theta}(f)$ and study some interesting results.\\


AMS Subject Classification No (2010): Primary 40A05; Secondary 46A70, 40A99, 46A99.
\end{abstract}
\maketitle
\pagestyle{myheadings}
\section{Introduction}
Throughout, $\mathbb{N},$ $\mathbb{R}$ and $\mathbb{C}$ will denote the set of natural, real and complex numbers, respectively and $x=(x_{k})$ denotes a sequence  whose $k^{th}$ term is $x_{k}$. Similarly $w,c,\ell_{\infty}, \ell_1$ denotes the space of \textit{all, convergent, bounded, absolutely summable} sequences of complex terms, respectively.\\

We use the symbol $<m,n>$ to denote the greatest common divisor of two integers $m$ and $n.$ \\
W.H.Ruckle \cite{Ruckle12}, introduced the notions arithmetic convergence as a sequence $x=(x_m)$ is called \textit{arithmetically convergent} if for each $\varepsilon > 0$ there is an integer $n$ such that for every integer $m$ we have $\left|x_m- x_{<m,n>}\right|< \varepsilon.$ We denote the sequence space of all arithmetic convergent sequence by $AC$. The studies on arithmetic convergence and related results  can be found in \cite{Ruckle12, hazarika2, hazarika3, hazarika4,tajahazacakalli}.\\

The notion of a modulus function was introduced in $1953$ by Nakano \cite{nakano21}.
We recall \cite{maddox1,fisher1} that a modulus $f$ is a function $f:[0,\infty)\rightarrow [0,\infty)$ such that
\begin{itemize}
\item[(i)] $f(x) = 0$ if and only if $x = 0$,
\item[(ii)]  $f(x + y)\leq f(x) + f(y)$ for all $x \geq 0, y\geq 0,$
\item[(iii)] $f$ is increasing,
\item[(iv)] $f$ is continuous from the right at $0$.
\end{itemize}
Because of (ii), $|f(x)-f(y)|\leq f(|x-y|)$ so that in view of (iv), $f$ is continuous everywhere on $[0,\infty)$. A modulus may be unbounded (for example,
$f(x) = x^p, 0 < p\leq 1$) or bounded (for example, $f(x) = \frac{x}{1+x}$).\\
It is easy to see that $f_1 + f_2$ is a modulus function when $f_1$ and $f_2$ are modulus functions, and that the function $f_i$($i$ is a positive integer), the
composition of a modulus function $f$ with itself $i$ times, is also a modulus function.\\
Ruckle \cite{ruckle} used the idea of a modulus function $f$ to construct a class of FK spaces
\begin{equation*}
X(f)=\left\{x=(x_k):\sum_{k=1}^{\infty}f(\left|x_k\right|)<\infty \right\}.
\end{equation*}
The space $X(f)$ is closely related to the space $l_1$ which is an $X(f)$ space with $f(x) = x$ for all real $x\geq 0.$
Thus Ruckle \cite{ruckle} proved that, for any modulus $f,$  $X( f )\subset l_1.$ 
The space $X( f )$ is a Banach space with respect to the norm $\left\|x\right\|= \sum_{k=1}^\infty f\left(\left|x_k \right|\right)< \infty.$\\
Spaces of the type $X( f )$ are a special case of the spaces structured by Gramsch in \cite{Gramsch1}.
By a lacunary sequence we mean an increasing integer sequence $\theta =(k_r)$ such that $h_r=k_r-k_{r-1}\rightarrow \infty$ as $r\rightarrow \infty$. In this paper the intervals determined by $\theta$ will be denoted by $I_{r}=(k_{r-1},k_r]$ and also the ratio $\frac{k_r}{k_{r-1}},r\geq1, k_0\neq 0$ will be denoted by $q_r$. The space of lacunary convergence sequence $N_{\theta}$ was defined by Freedman \cite{freedman1} as follows:
\begin{equation*}
N_{\theta}=\left\{x=(x_i)\in w: \lim_{r\rightarrow \infty}\frac{1}{h_r}\sum_{i\in I_r}\left|x_i-l\right|=0 ~\text{for some}~ l\right\}.
\end{equation*}
The space $N_{\theta}$ is a $BK$-space with the norm
\begin{equation*}
\left\|x\right\|_{N_{\theta}}=\sup_r\frac{1}{h_r}\sum_{i\in I_r}\left|x_i\right|.
\end{equation*}
The notion of lacunary convergence has been investigated by \c{C}olak \cite{colak1}, Fridy and Orhan \cite{freedy1,freedy2}, Tripathy and Et \cite{tripathy2} and many others in the recent past.\\

The main purpose of this paper is to introduce and study the concept of lacunary arithmetic convergence.
\section{Lacunary Arithmetic Convergence}
In this section we introduce the lacunary arithmetic convergent sequence space $AC_{\theta}$ as follows:\\
\begin{equation*}
AC_{\theta}=\left\{(x_m):\lim_{r\rightarrow \infty}\frac{1}{h_r}\sum_{m\in I_r}\left|x_m-x_{<m,n>}\right|=0 ~\text{for some integer}~n \right\}.
\end{equation*}

\begin{thm}
The sequence space $AC_{\theta}$ is linear. 
\end{thm}
\begin{proof}
Let $(x_m)$ and $(y_m)$ be two sequences in $AC_{\theta}$. Then for an integer $n$
\begin{equation*}
\lim_{r\rightarrow \infty}\frac{1}{h_r}\sum_{m\in I_r}\left|x_m-x_{<m,n>}\right|=0~\text{and}~ \lim_{r\rightarrow \infty}\frac{1}{h_r}\sum_{m\in I_r}\left|y_m-y_{<m,n>}\right|=0 .
\end{equation*}
Let $\alpha$ and $\beta$ be two scalars, then there exist integers $T_{\alpha}$ and $M_{\beta}$ such that $\left|\alpha\right|\le T_{\alpha}$ and $\left|\beta\right|\le M_{\beta}$. Thus
\begin{align*}
&\frac{1}{h_r}\sum_{m\in I_r}\left|\alpha x_m+ \beta y_m -(\alpha x_{<m,n>}+ \beta y_{<m,n>})\right|\\
&\le T_{\alpha} \frac{1}{h_r}\sum_{m\in I_r}\left|x_m-x_{<m,n>}\right|+ M_{\beta}\frac{1}{h_r}\sum_{m\in I_r}\left|y_m-y_{<m,n>}\right|
\end{align*}  
which implies that $\alpha x_m+ \beta y_m \rightarrow \alpha x_{<m,n>}+ \beta y_{<m,n>}$.\\
Hence $AC_{\theta}$ is linear.                                                                                                                                     \end{proof} 

\begin{thm}
If $(x_m)$ is a sequence in $AC$ then $(x_m)$ is a sequence in $ AC_\theta$.
\end{thm}
\begin{proof}
Let $(x_m)$ be a sequence in $AC$. Then for $\varepsilon>0$ there is an integer $n$ such that
\begin{equation*}
\left|x_m-x_{<m,n>}\right|< \varepsilon.
\end{equation*}
Now, for an integer $n$, we have
\begin{eqnarray*}
\frac{1}{h_r}\sum_{m\in I_r}\left|x_m-x_{<m,n>}\right| &=&\frac{1}{h_r}\left[\sum_{m=1}^{k_r}\left|x_m-x_{<m,n>}\right|-\sum_{m=1}^{k_{r-1}}\left|x_m-x_{<m,n>}\right|\right]\\
&<& \frac{1}{h_r}(h_r \varepsilon) \\
&=& \varepsilon.
\end{eqnarray*}
Thus $(x_m)\in AC_\theta$.
\end{proof}

\begin{definition}\cite{freedman1}
Let $\theta=(k_r)$ be a lacunary sequence. A lacunary refinement of $\theta$ is a lacunary sequence $\theta'=(k'_r)$ satisfying $(k_r)\subseteq (k'_r)$.
\end{definition}

\begin{thm}
If $\theta'$ is a lacunary refinement of a lacunary sequence $\theta$ and $(x_m)\in AC_{\theta'}$ then $AC_{\theta}$. 
\end{thm}
\begin{proof}  
Suppose for each $I_r$ of $\theta$ contains the point $(k'_{r,t})_{t=1}^{\eta(r)}$ of $\theta'$ such that 
\begin{equation*}
k_{r-1}<k'_{r,1}<k'_{r,2}<\ldots <k'_{\eta,\eta(r)}=k_r, 
\end{equation*} 
where $I'_{r,t}=\left(I'_{r,t-1},I'_{r,t}\right]$.\\
Since $(k_r)\subseteq (k'_r)$, so $r,\eta(r)\geq1$.\\
Let $(I^*)_{j=1}^\infty$ be the sequence of interval $(I^*_{r,t})$ ordered by increasing right end points. Since $(x_m)\in AC_{\theta'}$, then for each $\varepsilon>0$,
\begin{equation*}
\frac{1}{h_j^*}\sum_{I_j^*\subset I_r}\left|x_m-x_{<m,n>}\right|< \varepsilon.
\end{equation*}  
Also, since $h_r=k_r-k_{r-1}$, so $h'_{r,t}=k'_{r,t}-k'_{r,t-1}$.\\
For each $\varepsilon>0$,
\begin{equation*}
\frac{1}{h_r}\sum_{m\in I_r}\left|x_m-x_{<m,n>}\right|\leq \frac{1}{h_j^*}\sum_{I_j^*\subset I_r}\left|x_m-x_{<m,n>}\right|< \varepsilon.
\end{equation*}
This implies $(x_m)\in AC_{\theta}$.
\end{proof}

Based on the idea of strongly Ces\`{a}ro summable sequences and arithmetic convergent sequences, we introduce a new sequence space $AC_{\sigma_1}$ defined as follows:
\begin{equation*}
AC_{\sigma_1}=\left\{(x_m): \text{there exists an integer}~ n~ \text{such that}~ \frac{1}{t}\sum_{m=1}^{t}\left|x_m-x_{<m,n>}\right|\rightarrow 0~\text{as}~t\rightarrow \infty\right\}.
\end{equation*}
Functional analytic studies of the space $\left|\sigma_1\right|$ of strongly Ces\`{a}ro summable sequences, and other closely related spaces can be found in \cite{borwein1,maddox2}.\\
In this section we shall mostly focus on the connection between the spaces $AC_{\theta}$ and $AC_{\sigma_1}.$

\begin{thm}
The sequence space $AC_{\sigma_1}$ is a linear space.
\end{thm}
\begin{proof}
The proof is a routine exercise and hence ommited.
\end{proof}

\begin{thm}\label{thm1}
Let $\theta=(k_r)$ be a lacunary sequence. If $\liminf q_r > 1$ then $AC_{\sigma_1}\subseteq AC_{\theta}.$
\end{thm}
\begin{proof}
Let $(x_m)\in AC_{\sigma_1}$ and $\liminf q_r > 1$. Then there exists $\delta>0$ such that $q_r=\frac{k_r}{k_{r-1}} \geq 1+\delta$ for sufficiently large $r.$ We can also choose a sufficiently large $r$ so that $\frac{k_r}{h_r}\leq \frac{1+\delta}{\delta}.$ Then
\begin{align*}
&\frac{1}{k_r}\sum_{m=1}^{k_r}\left|x_m-x_{<m,n>}\right|\\
 &\geq  \frac{1}{k_r}\sum_{m\in I_r}\left|x_m-x_{<m,n>}\right| \\
&= \frac{h_r}{k_r}\left(h_r^{-1}\sum_{m\in I_r}\left|x_m-x_{<m,n>}\right|\right) \\
&\geq  \frac{\delta}{1+\delta} \left(h_r^{-1}\sum_{m\in I_r}\left|x_m-x_{<m,n>}\right|\right)
\end{align*}
which proves that $(x_m)\in AC_{\theta}.$
\end{proof}

\begin{thm}\label{thm2}
For $\limsup q_r< \infty,$ we have $AC_{\theta}\subseteq AC_{\sigma_1}.$
\end{thm}
\begin{proof}
Let $\limsup q_r< \infty$ then there exists $K>0$ such that $q_r< K$ for every $r.$ Now for $\varepsilon>0$ and $(x_m)\in AC_{\theta}$ there exists $R$ such that for every $r\geq R,$
\begin{equation*}
\tau_r= \frac{1}{h_r}\sum_{m\in I_r}\left|x_m-x_{<m,n>}\right|< \varepsilon.
\end{equation*}
We can also find $T>0$ such that $\tau_r\leq T$ $\forall r.$ Let $t$ be any integer with $k_r\geq t\geq k_{r-1}.$ Then
\begin{align*}
&\frac{1}{t}\sum_{m=1}^{t}\left|x_m-x_{<m,n>}\right|\\
&\leq  \frac{1}{k_{r-1}}\sum_{m=1}^{k_r}\left|x_m-x_{<m,n>}\right|\\
&= \frac{1}{k_{r-1}}\sum_{i=1}^{R}\sum_{m\in I_i}\left|x_m-x_{<m,n>}\right|+ \frac{1}{k_{r-1}}\sum_{i=R+1}^{k_r}\sum_{m\in I_i}\left|x_m-x_{<m,n>}\right|\\
&\leq  \frac{1}{k_{r-1}}\sum_{i=1}^{R}\sum_{m\in I_i}\left|x_m-x_{<m,n>}\right|+ \frac{1}{k_{r-1}}\left(\varepsilon(k_r-k_R)\right)\\
&\leq  \frac{1}{k_{r-1}}\sum_{i=1}^{R}h_i\tau_i+ \frac{1}{k_{r-1}}\left(\varepsilon(k_r-k_R)\right)\\
&\leq  \frac{1}{k_{r-1}}\left(\sup_{i\leq R}\tau_ik_R\right)+ \varepsilon K \\
&< \frac{k_R}{k_{r-1}}T+\varepsilon K
\end{align*}
from which we deduce that $(x_m)\in AC_{\sigma_1} .$
\end{proof}
The following corollary follows from Threoms \ref{thm1} and \ref{thm2}.

\begin{corollary} If
$1< \liminf q_r \leq \limsup q_r < \infty,$
then $AC_{\theta}=AC_{\sigma_1}$.
\end{corollary}

Next we introduce the lacunary arithmetic convergent sequence space $AC_{\theta}(f)$ defined by modulus function $f$.\\
Let $f$ be a modulus function. We define
\begin{equation*}
AC_{\theta}(f)=\left\{(x_m):\text{there exists an integer} ~n~ \text{such that}~ \frac{1}{h_r}\sum_{m\in I_r}f(\left|x_m-x_{<m,n>}\right|)\rightarrow 0~ \text{as}~ r\rightarrow \infty\right\}.
\end{equation*}
Note that if we put $f(x)=x$, then $AC_{\theta}(f)=AC_{\theta}.$
\begin{thm}
The sequence space $AC_{\theta}(f)$ is a linear space.
\end{thm}
\begin{proof}
Let $(x_m)$ and $(y_m)$ be two sequences in $AC_{\theta}(f).$ Then for an integer $n$ and $\varepsilon>0,$
\begin{equation*}
\frac{1}{h_r}\sum_{m\in I_r}f(\left|x_m-x_{<m,n>}\right|)\rightarrow 0~\text{and}~ \frac{1}{h_r}\sum_{m\in I_r}f(\left|y_m-y_{<m,n>}\right|)\rightarrow 0 ~\text{as}~r\rightarrow \infty.
\end{equation*}
Let $\alpha$ and $\beta$ be two scalars, then there exist integers $T_{\alpha}$ and $M_{\beta}$ such that $\left|\alpha\right|\le T_{\alpha}$ and $\left|\beta\right|\le M_{\beta}.$ Thus
\begin{align*}
&\frac{1}{h_r}\sum_{m\in I_r}f(\left|\alpha x_m+ \beta y_m -(\alpha x_{<m,n>}+ \beta y_{<m,n>})\right|)\\
&\leq T_{\alpha} \frac{1}{h_r}\sum_{m\in I_r}f(\left|x_m-x_{<m,n>}\right|)+ M_{\beta}\frac{1}{h_r}\sum_{m\in I_r}f(\left|y_m-y_{<m,n>}\right|)\\
&\rightarrow  0 ~\text{as}~ r\rightarrow \infty 
\end{align*}  
which implies that $\alpha x_m+ \beta y_m \rightarrow \alpha x_{<m,n>}+ \beta y_{<m,n>}$ in $AC_\theta (f)$.\\
Hence $AC_{\theta}(f)$ is linear.
\end{proof}

\begin{proposition}\label{fisher}\cite{fisher1}
Let $f$ be a modulus and let $0<\delta<1$. Then for each $x\geq \delta$, we have $f(x)\leq 2f(1)\delta^{-1}x$.
\end{proposition}

\begin{thm}
Let $f$ be any modulus such that $\lim_{t\rightarrow \infty}\frac{f(t)}{t}= \beta > 0$ then $AC_{\theta}(f)=AC_{\theta}$. 
\end{thm}
\begin{proof}
Let $(x_m)\in AC_{\theta}$, then for an integer $n$,\\
\begin{equation*}
\tau_r= \frac{1}{h_r}\sum_{m\in I_r}\left|x_m-x_{<m,n>}\right|\rightarrow 0 ~\text{as}~r\rightarrow \infty.
\end{equation*}
Let $\varepsilon>0$ be given. We choose $0<\delta<1$ such that $f(u)<\varepsilon$ for every $u$ with $0\leq u\leq \delta$. We can write
\begin{align*}
&\frac{1}{h_r}\sum_{m\in I_r}f(\left|x_m-x_{<m,n>}\right|)\\
&= \frac{1}{h_r}\sum_{\underset {\left|x_m-x_{<m,n>}\right| \leq \delta}{m\in I_r;}}f(\left|x_m-x_{<m,n>}\right|)+ \frac{1}{h_r}\sum_{\underset {\left|x_m-x_{<m,n>}\right| > \delta}{m\in I_r;}}f(\left|x_m-x_{<m,n>}\right|) \\
&\leq \frac{1}{h_r}(h_r \varepsilon) + 2f(1)\delta^{-1}\tau_r, ~\text{using (Proposition \ref{fisher})}\\
&\rightarrow  0 ~ \text{as}~ r\rightarrow \infty.
\end{align*}
Therefore $(x_m)\in AC_{\theta}(f)$.\\
Till this part of the proof we do need $\beta >0$. Now let $\beta>0$ and let $(x_m)\in AC_{\theta}(f)$. Since $\beta>0$, we have $f(t)\geq \beta t$ $\forall t \geq 0$. Hence it follows that $(x_m)\in AC_{\theta}$.  
\end{proof}



\small
\begin{thebibliography}{99}
\bibitem{borwein1} D. Borwein, Linear functionals connected with strong Ces\`{a}ro summability, J. Lond. Math. Soc. 40(1965) 628-634.

\bibitem{colak1} R. \c{C}olak, Lacunary strong convergence of difference sequences with respect to a modulus function, Filomat 17(2003) 9-14.
\bibitem{freedman1} A. R. Freedman, J. J. Sember, M. Raphael, Some Ces$\grave{a}$ro-type summability spaces, Proc. Lond. Math. Soc.  37(1978) 508-520.
\bibitem{freedy1} J. A. Fridy, C. Orhan, Lacunary statistical summability, J. Math. Anal. Appl.  173(2)(1993) 497-504.
\bibitem{freedy2} J. A. Fridy, C. Orhan, Lacunary statistical convergence, Pacific J. Math. 160(1)(1993) 43-51.
\bibitem{Gramsch1} B. Gramsch, Die Klasse metrisher linearer R\"{a}ume $\mathcal{L}_{\phi},$ Math. Ann. 171(1)(1967) 61-78.



\bibitem {maddox1} I. J. Maddox, Sequence spaces defined by a modulus, Math. Proc. Camb. Philos. 
Soc. 100 (1986) 161–166.
\bibitem{maddox2} I.J. Maddox, Spaces of strongly summably sequences, Quart. J. Math. Oxford (2)18 (1967) 345-55.

\bibitem {nakano21}H. Nakano, Concave modulars, J. Math. Soc. Japan 5 (1953) 29–49.

\bibitem{fisher1} S. Pehlivan, B. Fisher, On Some Sequence Space, Indian J. Pure Appl. 25(10)(1994) 1067-1071.

\bibitem {Ruckle12} W. H. Ruckle, Arithmetical Summability, J. Math. Anal. Appl. 396(2012) 741-748.

\bibitem {ruckle} W. H. Ruckle, FK spaces in which the sequence of coordinate vectors is bounded, Canad. J. Math. 25 (1973), 973-978.


\bibitem{hazarika2} Taja Yaying and Bipan Hazarika, On arithmetical summability and multiplier sequences, National Academy Sci. Letters (accepted).
\bibitem{hazarika3}  Taja Yaying and Bipan Hazarika, On arithmetic continuity, Bol. Soc. Parana. Mat. (in press)
\bibitem{hazarika4}  Taja Yaying and Bipan Hazarika, Some new types of continuity in asymmetric metric space, Contem. Anal. Appl. Math. (accepted)
\bibitem{tajahazacakalli} Taja Yaying, Bipan Hazarika and Huseyin \c{C}akalli, New results in quasi cone metric spaces, J. Math. Comput. Sci. (accepted) 
\bibitem{tripathy2} B. C. Tripathy, M. Et, On generalized difference lacunary statistical convergence, Studia Univ. Babes-Bolyai Math. 50(1)(2005) 119-130.
\end{thebibliography}
\end{document}